\documentclass[12pt]{amsart}
\usepackage[margin=1in]{geometry}
\usepackage{bbm}
\usepackage{dsfont}
\usepackage{relsize}
\usepackage{amsthm}
\usepackage[breaklinks]{hyperref}
\usepackage{autonum}
\usepackage{blindtext}
\usepackage{amsmath,amssymb}
\usepackage[usenames,dvipsnames,svgnames]{xcolor}
\usepackage{mathrsfs}
\usepackage{csquotes}
\usepackage{mathtools}
\usepackage{tcolorbox}
\usepackage{enumitem}
\usepackage[]{hyperref}
\hypersetup{colorlinks=true,linkcolor=WildStrawberry,citecolor=cyan}
\newtheorem{thm}{Theorem}[section]

\newtheorem{lem}[thm]{Lemma}

\theoremstyle{hypothesis}

\theoremstyle{definition}
\newtheorem{defn}[thm]{Definition}
\theoremstyle{remark}
\newtheorem{rem}[thm]{Remark}
\theoremstyle{example}

\numberwithin{equation}{section}
\makeatletter
\renewcommand\paragraph{\@startsection{paragraph}{4}{\z@}%
	{-2.5ex\@plus -1ex \@minus -.25ex}%
	{1.25ex \@plus .25ex}%
	{\normalfont\normalsize\centering\bfseries}}
\makeatother
\renewcommand\paragraph{\@startsection{paragraph}{4}{\z@}%
	{-2.5ex\@plus -1ex \@minus -.25ex}%
	{1.25ex \@plus .25ex}%
	{\normalfont\normalsize\centering\bfseries}}
\newcommand{\To}{\longrightarrow}
\makeatother
\begin{document}
	\title[Fully nonlinear degenerate equations with applications to Grad equations]
	{FULLY NONLINEAR DEGENERATE EQUATIONS WITH APPLICATIONS TO GRAD EQUATIONS}
	\author[P.\,Oza]
	{ Priyank Oza}
	\address{Priyank\, Oza \hfill\break
		Indian Institute of Technology Gandhinagar \newline
		Palaj, Gandhinagar Gujarat, India-382355.}
	\email{priyank.k@iitgn.ac.in, priyank.oza3@gmail.com}
	\thanks{Submitted \today.  Published-----.}
	\subjclass[2021]{35J25, 35J60, 35J70, 35D40}
	\keywords{Fully nonlinear degenerate elliptic equations, viscosity solution, Pucci's extremal operator, Dirichlet boundary value problem}
	
	\begin{abstract}
		We consider a class of degenerate elliptic fully nonlinear equations with applications to Grad equations: 
		\begin{align}
			\begin{cases}
				|Du|^\gamma \mathcal{M}_{\lambda,\Lambda}^+\big(D^2u(x)\big)=f\big(|u\geq u(x)|\big) &\text{ in }\Omega\\
				u=g &\text{ on }\partial\Omega,
			\end{cases}
		\end{align}
		where $\gamma\geq 1$ is a constant, $\Omega$ is a bounded domain in $\mathbb{R}^N$ with $C^{1,1}$ boundary. We prove the existence of a $W^{2,p}$-viscosity solution to the above equation, which degenerates when the gradient of the solution vanishes.
	\end{abstract}
	
	\maketitle

	\section{\textbf{Introduction}}
   We study the following degenerate problem:
	\begin{align}\label{eq 1.1}
		\begin{cases}
			|Du|^\gamma \mathcal{M}_{\lambda,\Lambda}^+\big(D^2u(x)\big)=f\big(|u\geq u(x)|\big) &\text{ in }\Omega\\
			u=g &\text{ on }\partial\Omega,
		\end{cases}
	\end{align}
	where $\gamma\geq 1$ is a constant, $\Omega$ is a bounded domain in $\mathbb{R}^N$ with $C^{1,1}$ boundary, $|.|$ denotes the Lebesgue measure in $\mathbb{R}^N,$ $f : [0,|\Omega|]\To \mathbb{R}$ is a non-decreasing, non-negative continuous function and $u:\Omega\To \mathbb{R}.$ Here, $\mathcal{M}_{\lambda,\Lambda}^+$ is the Pucci's extremal operator. In our setting, by $u\geq u(x),$ we mean,
	\begin{align}
		\{\omega\in\Omega: u(\omega)\geq u(x)\}
	\end{align}
	called the superlevel sets of $u.$ We establish the existence of a $W^{2,p}$-viscosity solution (also known as $L^p$-viscosity solution) to \eqref{eq 1.1}. We mention that the notion of $W^{2,p}$-viscosity solution was defined by Caffarelli et al. \cite{Caffar}. The case when $\gamma=0$ in \eqref{eq 1.1}, the existence of a $W^{2,p}$-viscosity solution is established by L. Caffarelli and I. Tomasetti \cite{Caffar1}.
		
	H. Grad \cite{Grad} was the first to introduce such equations, which appear in plasma physics, called \enquote{Grad equations}. More specifically, the author studied the following equation in $3$-dimension:
	\begin{align}
		\Delta \Psi=F(V,\Psi,\Psi',\Psi''),
	\end{align}
	where the right hand side (RHS) is a second order differential operator in $\Psi(V)$ for a surface $\Psi=$ constant. Here, $\Psi'(V)$ denotes the derivative with respect to the volume and $\Psi(V)$ denotes the inverse function to $V(\Psi),$ the volume within $\Psi.$ The author further pointed out that plasma equations could be further simplified using $u^*$ given by
	\begin{align}
		u^*(t)\coloneqq \inf\bigg\{s:|u<s|\geq t\bigg\}.
	\end{align} 
	These also appear as Queer Differential Equations in the literature and have a wide range of applications. For instance, they appear in modeling plasma, specifically to study plasma confined in a toroidal container. The equations also have connections to financial mathematics. Further, these equations have been widely studied by several researchers. R. Temam \cite{Temam} was the first in this direction to investigate the problem of the form \eqref{eq 1.1} for the Laplacian. The author established the existence of a solution to 
	\begin{align}
		\Delta u=g\big(|u<u(x)|,u(x)\big)+f(x)
	\end{align}
exploiting the properties of directional derivatives of $u^*.$
	One can also see J. Mossino and R. Temam \cite{Temam2}, and also P. Laurence and E. Stredulinsky \cite{Laure1, Laure2} as well as the references therein for a brief introduction.
	
	We point out that in all of the above-mentioned research works, the problem was studied using variational methods. In a recent paper, L. Caffarelli and I. Tomasetti \cite{Caffar1} studied the equation similar to J. Mossino and R. Temam \cite{Temam2} for fully nonlinear uniformly elliptic operators using the viscosity approach. More precisely, they considered the following problem:
	\begin{align}\label{eq caf}
		\begin{cases}
			F\big(D^2u(x)\big)=f\big(|u\geq u(x)|\big) &\text{ in }\Omega\\
			u=g &\text{ on }\partial\Omega,
		\end{cases}
	\end{align}
	for convex, uniformly elliptic operator $F.$ Authors proved the existence of a $W^{2,p}$-viscosity solution, $u$ to \eqref{eq caf} satisfying the following estimate:
	\begin{align}
		\|u\|_{W^{2,p}(\Omega)}\leq C\big[\|u\|_{\infty,\Omega}+\|g\|_{W^{2,p}(\Omega)}+\|f\big(|u\geq u(x)|\big)\|_{p,\Omega}\big].
	\end{align}

We refer to \cite{Felmer 18, Felmer 19, Oza Hopf, Quaas, Quaas 2, Tyagi 3, Tyagi 1, Tyagi 2} for the existence and qualitative questions pertaining to extremal Pucci's equations.
On the other hand, equations concerning gradient degenerate fully nonlinear elliptic operators have been investigated widely in the last decade. The pioneering works in this direction are due to I. Birindelli and F. Demengel. Authors proved several important results for these operators in a series of papers: comparison principle and Liouville-type results \cite{Birin 1}, regularity and uniqueness of eigenvalues and eigenfunctions \cite{Birin 3, Birin 2}, $C^{1,\alpha}$ regularity in the radial case \cite{Birin 4}. Further, the equations of the form
\begin{align}\label{eq Imb}
	|Du|^\gamma F(D^2u)=f \text{ in }B_1,
\end{align}
when $\gamma\geq 0$ is a constant, $f\in L^\infty(B_1,\mathbb{R}),$ were investigated by C. Imbert and L. Silvestre \cite{Imbert}. More precisely, the authors established interior $C^{1,\alpha}$ regularity of solutions for equations of the form \eqref{eq Imb}. One may also see \cite{Oza Gradient Degeneracy} for variable exponent degenerate mixed fully nonlinear local and nonlocal equations.
	\noindent Motivated by the above works and recently by the work of L. Caffarelli $\&$ I. Tomasetti \cite{Caffar1}, it is natural to ask the following question:\\
	
	\noindent\textbf{\textit{Question:}} \textit{Do we have the existence of a viscosity solution to \eqref{eq 1.1}?}\\
	
	The aim of this paper is to answer this question affirmatively. The crucial difference to our problem from \cite{Caffar1} is due to the fact that $|Du|^\gamma F(D^2u)$ degenerates along the set of critical points, $\mathcal{C}\coloneqq\{x:Du(x)=0\}.$ The problem is challenging due to the following reasons:
	\begin{enumerate}[label=\upshape{(C\arabic*)}, ref=(C\arabic*)]
		\item\label{C1} The RHS of \eqref{eq 1.1} is a function of measure of superlevel sets. This makes the problem nonlocal.
		\item\label{C2} The LHS of \eqref{eq 1.1} is degenerate. The fundamental theory of $L^p$-viscosity solutions does not work directly here since it requires the uniform ellipticity of the operator. Also, when $f\in C(\Omega),$ the problem can be discussed in the $C$-viscosity sense but in the case of discontinuous data, when $f\in L^p(\Omega),$ the problem needs to be treated in the $L^p$-viscosity sense. We point out that this situation occurs while approximating the RHS of \eqref{eq 1.1}. 
	\end{enumerate} 
	
	We use the $L^p$-viscosity solution approach for Monge-Amp\'ere equation as in \cite{Brand, Caffar1} to \eqref{eq 1.1}. To handle the above mentioned challenges, we first consider the following approximate problem:
	\begin{align}\label{eq 1.1'}
		\begin{cases}
			|Du|^\gamma \mathcal{M}_{\lambda,\Lambda}^+\big(D^2u(x)\big)+\mathlarger{\mathlarger\varepsilon}\Delta u=f\big(|u\geq u(x)|\big) &\text{ in }\Omega\\
			u=g &\text{ on }\partial\Omega,
		\end{cases}
	\end{align}
     for $\mathlarger{\mathlarger\varepsilon}>0.$ 
Further, using the approximations in the RHS of the equation and exploiting the results available for uniform elliptic operators, for instance, Theorem \ref{Winter} and Theorem \ref{Caffar} (see next), we establish the existence of a viscosity solution to the approximate problem \eqref{eq 1.1'}. This yields the existence of a viscosity solution to \eqref{eq 1.1}. More precisely, using the idea of Amadori et al. \cite{Brand}, we first get the existence of a $W^{2,p}$-viscosity solution to the approximate problem \eqref{eq 1.1'} by invoking Theorem 2.1 \cite{Caffar1}. We recall that the estimate established in \cite{Caffar1} is not adequate to claim the uniform bound on the $W^{2,p}$-viscosity solution of \eqref{eq 1.1'}. To show the existence of a solution to the original problem \eqref{eq 1.1}, we seek the uniform bound on the solutions of \eqref{eq 1.1'}, which is crucial in approaching $\mathlarger{\mathlarger\varepsilon}\To 0^+.$ 
We invoke the Alexandroff-Bakelman-Pucci (ABP) estimates from Caffarelli et al. \cite{Caffar} to sort this issue. These estimates play a crucial role in obtaining uniform bounds on the $W^{2,p}$-viscosity solutions to \eqref{eq 1.1'}. 
	 
Throughout the paper, we consider $\Omega$ to be a bounded $C^{1,1}$ domain in $\mathbb{R}^N,$ $N\geq 2.$ 
%
The main result of this paper is the following:
	\begin{thm}\label{Main1}
		Let $\gamma\geq 1$ be a constant. Let $\Omega\subset\mathbb{R}^N$ be a bounded $C^{1,1}$ domain. Let $f \in C\big([0,|\Omega|],\mathbb{R}\big)$ be a non-decreasing, non-negative function and $g\in W^{2,p}(\Omega)\cap C(\overline{\Omega}),$ $p>N.$ Consider the problem
		\begin{align}\label{eq Main}
			\begin{cases}
				|Du|^\gamma \mathcal{M}_{\lambda,\Lambda}^+\big(D^2u(x)\big)=f\big(|u\geq u(x)|\big) &\text{ in }\Omega\\
				u=g &\text{ on }\partial\Omega.
			\end{cases}
		\end{align}
		Then there exists a $W^{2,p}$-viscosity solution of (\ref{eq Main}). Moreover, the solution satisfies the following estimate:
		\begin{align}
			\|u\|_{W^{2,p}(\Omega)}\leq C\bigg(\|u\|_{\infty,\Omega}+\|g\|_{W^{2,p}(\Omega)}+f(|u\geq u(x)|)\bigg),  
		\end{align}
		where $C>0$ is a constant.
	\end{thm}
	
	\begin{rem}
		By Sobolev embedding theorem we have that the solution is $C^{1,\alpha}(\overline{\Omega})$ regular for any $\alpha<1.$ 
	\end{rem}

	The organization of the paper is as follows. In Section 2, we recall the basic definitions and several key results used in the ensuing sections of the paper. Section 3 is devoted to the proof of our main result. Here, we sketch the plan of our proof:
	\begin{enumerate}
		\item [(i)] Approximate the left-hand side (LHS), i.e., the operator $|Du|^\gamma \mathcal{M}_{\lambda,\Lambda}^+\big(D^2u\big)$ by $|Du|^\gamma \mathcal{M}_{\lambda,\Lambda}^+\big(D^2u\big)+\mathlarger{\mathlarger\varepsilon}\Delta u$ for $\mathlarger{\mathlarger\varepsilon}>0.$
		\item [(ii)] Fix a Lipschitz function $v$ in the right hand side of \eqref{eq 1.1'}.
		\item [(iii)] Construct a sequence of $L^p$-functions converging to right hand side (for fixed Lipschitz function $v$) and obtain a sequence of solutions.
		\item [(iv)] Obtain the existence of solution to equation pertaining $|Du|^\gamma \mathcal{M}_{\lambda,\Lambda}^+\big(D^2u\big)+\mathlarger{\mathlarger\varepsilon}\Delta u$ for fixed Lipschitz function $v$ in the RHS.
		\item [(v)] Use Theorem 2.1 \cite{Caffar1} (an application of Schaefer fixed point Theorem) to show the existence of a solution to \eqref{eq 1.1'}.
		\item [(vi)] Establish the existence of a $W^{2,p}$-viscosity solution to \eqref{eq Main}.
	\end{enumerate}
	
	\section{\textbf{Preliminaries}}
We recall that a continuous mapping $F: S^N\To \mathbb{R}$ is \textit{uniformly elliptic} if:\\
For any $A\in \mathcal{S}^N,$ where $\mathcal{S}^N$ is the set of all $N\times N$ real symmetric matrices, there exist two positive constants $\Lambda\geq\lambda>0$ s.t.
\begin{align}\label{structure}
	\lambda\|B\|\leq F(A+B)-F(A)\leq N\Lambda\|B\| \text{ for all positive semi-definite }B\in\mathcal{S}^N,
\end{align} 
and $\|B\|$ is the largest eigenvalue of $B.$ Here, we have the usual partial ordering: $A\leq B$ in $\mathcal{S}^N$ means that $\langle A\xi,\xi\rangle \leq \langle B\xi, \xi\rangle$ for any $\xi \in \mathbb{R}^N.$ In other words, $B-A$ is positive semidefinite, where $\|B\|$ is the largest eigenvalue of $B.$

Let $S\in \mathcal{S}^N$ then for the given two parameters $\Lambda\geq\lambda>0,$ Pucci's maximal operator is defined as follows:
\begin{align}\label{M+}
	\mathcal{M}_{\lambda,\Lambda}^+(S)\coloneqq \Lambda\displaystyle{\sum_{e_i\geq 0}e_i}+\lambda\displaystyle{\sum_{e_i< 0}e_i
	},
\end{align}
where $\{e_i\}_{i=1}^N$ are the eigenvalues of $S.$ This operator is uniformly elliptic and subadditive, i.e.,
 \begin{align}
 	\mathcal{M}_{\lambda,\Lambda}^+(A+B)\leq \mathcal{M}_{\lambda,\Lambda}^+(A)+\mathcal{M}_{\lambda,\Lambda}^+(B),
 \end{align}
for $M,N\in\mathcal{S}^N.$ Clearly, for $\lambda=\Lambda=1,$ $\mathcal{M}_{\lambda,\Lambda}^+=\Delta,$ i.e., classical Laplace operator.	

Next, we recall the notion of a viscosity solution. M. G. Crandall and P.-L. Lions \cite{Crand} were the first to introduce the concept of a viscosity solution. Now, we recall the definition of \textit{continuous viscosity solution} to the following equation:
\begin{align}\label{eq 2.1}
	|Du|^\gamma F\big(D^2u(x)\big)=f \text{ in }\Omega,	 
\end{align}
	  for $f\in C(\Omega).$
	\begin{defn}\cite{Birin 1}
		Let $u:\overline{\Omega}\To\mathbb{R}$ be an upper semicontinuous (USC) function in $\Omega.$ Then $u$ is called a \textit{viscosity subsolution} of \eqref{eq 2.1} if 
		\begin{align}
			|D\varphi(x)|^\gamma F\big(D^2\varphi(x)\big)\geq f(x),	 
		\end{align}
	whenever $\varphi\in C^2(\Omega)$ and $x\in \Omega$ is a local maximizer of $u-\varphi$ with $D\varphi\neq \textbf{0}\in\mathbb{R}^N.$ 
	\end{defn}

	
	\begin{defn}\cite{Birin 1}
		Let $u:\overline{\Omega}\To\mathbb{R}$ be a lower semicontinuous (LSC) function in $\Omega.$ Then $u$ is called a \textit{viscosity supersolution} of \eqref{eq 2.1} if
				\begin{align}
			|D\psi(x)|^\gamma F\big(D^2\psi(x)\big)\leq f(x),	 
		\end{align}
		whenever $\psi\in C^2(\Omega)$ and $x\in \Omega$ is a local minimizer of $u-\psi$ with $D\psi\neq \textbf{0}\in\mathbb{R}^N.$ 
	\end{defn}
	\begin{defn}\cite{Birin 1}
		A continuous function $u$ is said to be a \textit{viscosity solution} to (\ref{eq 2.1}) if it is a supersolution as well as subsolution to (\ref{eq 2.1}).
	\end{defn}	
	
	\noindent Let $h\in L^p(\Omega),$ $g\in W^{2,p}(\Omega)\cap C(\overline{\Omega})$ for $p>N.$ Let us consider the problem
	\begin{align}\label{eq og}
		\begin{cases}
			|Du|^\gamma \mathcal{M}_{\lambda,\Lambda}^+(D^2u)=h &\text{ in }\Omega,\\
			u=g &\text{ on }\partial\Omega.
		\end{cases}
	\end{align} 
	We mention that the classical definition of $W^{2,p}$-viscosity solution can not be applied for (\ref{eq og}), due to the lack of uniform ellipticity. Consider the problem:
\begin{align}\label{eq app}
	\begin{cases}
		|Du|^\gamma \mathcal{M}_{\lambda,\Lambda}^+\big(D^2u\big)+\mathlarger{\mathlarger\varepsilon}\Delta u=h &\text{ in }\Omega\\
		u=g &\text{ on }\partial\Omega,
	\end{cases}
\end{align}
for $p\in\mathbb{R}^N.$	
	Motivated by Caffarelli et al. \cite{Caffar} and Ishii et al. \cite{Ishii}, we define the $L^p$-viscosity subsolution (supersolution) to ($\ref{eq app}$) as follows: 

	\begin{defn} \label{ND}
		Let $u$ be an USC (respectively, LSC) function on $\overline{\Omega}.$ We say that $u$ is an $L^p$-viscosity subsolution (respectively, supersolution) to $\eqref{eq app}$ if $u\leq g$ (resp., $u\geq g$) on $\partial\Omega$ and $F(D^2u)+h\leq 0$ (resp., $\geq 0$) in $L^p$-viscosity sense, i.e., for all $\phi\in W^{2,p}(\Omega),$
		\begin{align}
			\text{ess }\displaystyle{\liminf_{x\To y}}\big(|D\phi(x)|^\gamma \mathcal{M}_{\lambda,\Lambda}^+\big(D^2\phi(x)\big)+\mathlarger{\mathlarger\varepsilon}\Delta \phi(x)-h(x)\big)\geq 0
		\end{align}
		\begin{align}
			\bigg(\text{resp., ess }\displaystyle{\limsup_{x\To y}}\big(|D\phi(x)|^\gamma \mathcal{M}_{\lambda,\Lambda}^+\big(D^2\phi(x)\big)+\mathlarger{\mathlarger\varepsilon}\Delta \phi(x)-h(x)\big)\leq 0\bigg),
		\end{align}
		for $y\in \Omega,$ the point of local maxima (respectively, minima) to $u-\phi.$ 
	\end{defn}
	We say that any continuous function $u$ is an $L^p$-viscosity solution to $(\ref{eq app})$ if it is both $L^p$-viscosity subsolution and supersolution to $(\ref{eq app}).$
	Now, we state a result concerning the existence and uniqueness of $W^{2,p}$-viscosity solution to the operator $F$ under certain hypotheses. The following result is due to N. Winter \cite{Winter}.
	\begin{thm}[Theorem 4.6 \cite{Winter}]\label{Winter}
		Let $\Omega$ be a bounded $C^{1,1}$ domain in $\mathbb{R}^N.$ Let $F(p,M)$ be a uniformly elliptic operator and convex in $M$-variable. Also, let $F(0,0)\equiv 0$ in $\Omega,$ $f\in L^p(\Omega)$ and $g\in W^{2,p}(\Omega)$ for $p>N.$ Then there exists a unique $W^{2,p}$-viscosity solution to
		\begin{align}\label{eq 2.9}
			\begin{cases}
				F(Du,D^2u)=f &\text{ in }\Omega\\
				u=g &\text{ on }\partial\Omega.
			\end{cases}
		\end{align}
		Moreover, $u \in W^{2,p}(\Omega)$ and 
		\begin{align}
			\|u\|_{W^{2,p}(\Omega)}\leq C\bigg(\|u\|_{\infty,\Omega}+\|g\|_{W^{2,p}(\Omega)}+\|f\|_{p,\Omega}\bigg),
		\end{align}
		for some positive constant $C.$
	\end{thm}
	
	\begin{thm}[Theorem 1.1 \cite{C1beta}]\label{Birin boundary}
		Let $\Omega$ be a bounded domain with $C^2$-boundary. Let $\gamma\geq 0$ and $F$ be a uniformly elliptic operator and $f\in C(\overline{\Omega}),$ $g\in C^{1,\beta}(\partial\Omega)$ for some $\beta\in(0,1).$ Then any viscosity solution $u$ of
		\begin{align}
			\begin{cases}
				|Du|^\gamma F\big(D^2u\big)=f &\text{ in }\Omega\\
				u=g \text{ on }\partial\Omega
			\end{cases}
		\end{align}
	is in $C^{1,\alpha}$ for some $\alpha=\alpha(\lambda,\Lambda,\|f\|_{\infty,\Omega},N,\Omega,,\beta).$ Moreover, $u$ satisfies the following estimate
	\begin{align}
		\|u\|_{C^{1,\alpha}(\overline{\Omega})}\leq C\left(\|g\|_{C^{1,\beta}(\partial\Omega)}+\|u\|_{\infty,\Omega}+\|f\|_{\infty,\Omega}^{\frac{1}{1+\gamma}}\right),
	\end{align}
for some positive constant $C=C(\alpha).$
	\end{thm}
	The following result plays an important role in Step 5 of the proof of our main result.
	\begin{thm}[Theorem 3.8 \cite{Caffar}]\label{Caffar}
		Let $F_i,$ $F$ be uniformly elliptic and $p>N.$ Let $f, f_i \in L^p(\Omega).$ Let $u_i\in C(\Omega)$ be $W^{2,p}$-viscosity subsolutions (supersolutions) to 
		\begin{align}
			F_i(D^2u_i)=f_i \text{ in }\Omega,
		\end{align}
		for $i=1,2,\dots.$ Assume that $u_i\To u$ locally uniformly in $\Omega.$ Also, assume that if for each $B(x_0,r)\subset\Omega$ and $g\in W^{2,p}(B(x_0,r)),$ we have
		\begin{align}
			\big\|\big(F_i(D^2u_i)-f_i(x)-F(D^2(u))+f(x)\big)^+\big\|_{p,B(x_0,r)}\To 0
		\end{align}
		\begin{align}
			\bigg(\big\|\big(F_i(D^2u_i)-f_i(x)-F(D^2(u))+f(x)\big)^-\big\|_{p,B(x_0,r)}\To 0\bigg).
		\end{align}
		Then $u$ is a $W^{2,p}$-viscosity subsolution (supersolution) to 
		\begin{align}
			F(D^2u)=f 	\text{ in } \Omega.
		\end{align}
	\end{thm}

	\section{\textbf{Proof of our main result}}
	\noindent \textbf{Proof of Theorem \ref{Main1}.}
	The original problem is
	\begin{align}\label{eq 3.4}
		\begin{cases}
			|Du|^\gamma F\big(D^2u(x)\big)=f\big(|u\geq u(x)|\big) &\text{ in }\Omega\\
			u=g &\text{ on }\partial\Omega.
		\end{cases}
	\end{align}
	
	\noindent \textbf{\tcbox{Step 1: Approximating the LHS}} 
	Consider the approximate problem:
	\begin{align}\label{eq v}
		\begin{cases}
			|Du|^\gamma \mathcal{M}_{\lambda,\Lambda}^+\big(D^2u\big)+\mathlarger{\mathlarger\varepsilon}\Delta u=f\big(|u\geq u(x)|\big) &\text{ in }\Omega\\
			u=g &\text{ on }\partial\Omega,
		\end{cases}
	\end{align}
for $\mathlarger{\mathlarger\varepsilon}>0.$ Since, $Gu:=|Du|^\gamma \mathcal{M}_{\lambda,\Lambda}^+\big(D^2u\big)+\mathlarger{\mathlarger\varepsilon}\Delta u$ is uniformly elliptic, so by Theorem 2.1 \cite{Caffar1}, we immediately have the existence of a $W^{2,p}$-viscosity solution (say $u_{\mathlarger\varepsilon}$) to \eqref{eq v} satisfying the following estimate:
	\begin{align}
		\|u_{\mathlarger\varepsilon}\|_{W^{2,p}(\Omega)}\leq C\bigg(\|u_{\mathlarger\varepsilon}\|_{\infty,\Omega}+\|g\|_{W^{2,p}(\Omega)}+\|f\big(|u\geq u(x)|\big)\|_{p,\Omega}\bigg).
	\end{align}
	By the above estimate, one can not directly claim the uniform bound on $u_{\mathlarger\varepsilon},$ which is crucial in order to pass the limit $\mathlarger{\mathlarger\varepsilon}\To 0$ to establish the existence of $W^{2,p}$-viscosity solution to \eqref{eq 3.4}. To overcome this difficulty, we further approximate problem \eqref{eq v}.\\

	\noindent \textbf{\tcbox{Step 2: Fixing a Lipschitz function in the RHS}}
	Next, following the arguments similar to \cite{Caffar1}, we fix a Lipschitz function $v$ in $\Omega,$ i.e., consider $h_v(x)\coloneqq f\big(|v\geq v(x)|\big)$ and reduce to the following problem:
	\begin{align}\label{eq propp}
		\begin{cases}
			|Du|^\gamma \mathcal{M}_{\lambda,\Lambda}^+\big(D^2u\big)+\mathlarger{\mathlarger\varepsilon}\Delta u=h_v &\text{ in }\Omega\\
			u=g &\text{ on }\partial\Omega.
		\end{cases}
	\end{align}
	
	\noindent \textbf{\tcbox{Step 3: Approximating RHS by a sequence of $L^p(\Omega)$ functions}}
	We consider a sequence of functions $\big\{h_v^{i}\big\}_{i=1}^\infty$ defined as
	\begin{align}
		h_v^{i}(x)\coloneqq f\bigg(i\int_0^{\frac{1}{i}}|v\geq v(x)-t|dt\bigg).
	\end{align}
	We approximate the function $$h_v(x)(=f\big(|v\geq v(x)|\big))$$ in the R.H.S. of \eqref{eq propp} by the sequence of functions $\big\{h_v^{i}\big\}_{i=1}^\infty.$ Hence, we have the following approximate problem:
	\begin{align}\label{eq 3.5}
		\begin{cases}
			|Du|^\gamma \mathcal{M}_{\lambda,\Lambda}^+\big(D^2u\big)+\mathlarger{\mathlarger\varepsilon}\Delta u=h^{i}_v &\text{ in }\Omega\\
			u=g &\text{ on }\partial\Omega,
		\end{cases}
	\end{align}
	for $i\geq 1.$
	Since $\big\{h^{i}_v\big\} \in L^p(\Omega).$ For each $i,$ by Theorem \ref{Winter}, we have the existence of a unique $W^{2,p}$-viscosity solution to (\ref{eq 3.5}).\\
	
%
%
	\begin{lem}\label{lem1}
		There exists a unique $W^{2,p}$-viscosity solution to \eqref{eq 3.5}. Moreover, it satisfies the following estimate:
		\begin{align}
			\|u^i_{\mathlarger\varepsilon}\|_{W^{2,p}(\Omega)}\leq C\bigg(\displaystyle{\max_{\partial \Omega}g}+\|g\|_{W^{2,p}(\Omega)}+f(|\Omega|)|\Omega|^{\frac{1}{p}}\bigg).
		\end{align}
	\end{lem}
\begin{proof}
By Theorem \ref{Winter}, we have the existence of a unique $W^{2,p}$-viscosity solution $u^i_{\mathlarger\varepsilon}$ to \eqref{eq 3.5} satisfying the following estimate:
\begin{align}\label{eq 4.9}
	\|u^i_{\mathlarger\varepsilon}\|_{W^{2,p}(\Omega)}\leq C\bigg(\|u^i_{\mathlarger\varepsilon}\|_{\infty,\Omega}+\|g\|_{W^{2,p}(\Omega)}+\|h^{i}_v\|_{p,\Omega}\bigg),
\end{align}
Also, it is easy to observe that
\begin{align}\label{es1}
	\|h_v^{i}\|_{\infty,\Omega}\leq f(|\Omega|),
\end{align}
and
\begin{align}\label{es2}
	\|h_v^{i}\|_{p,\Omega}&=\bigg(\int_\Omega|h_v^{i}(x)|^pdx\bigg)^{\frac{1}{p}}\\ 
	&\leq \|h_v^{i}\|_{\infty,\Omega}|\Omega|^{\frac{1}{p}}\\&\leq f(|\Omega|)|\Omega|^\frac{1}{p},
\end{align}
for each $i\geq 1.$ Thus the sequence of functions $h_v^{i}$ is uniformly bounded. Now, by ABP estimates established in \cite{Caffar} we have
\begin{align}
\displaystyle{\sup_\Omega}\,u^i_{\mathlarger\varepsilon}\leq \displaystyle{\sup_{\partial\Omega}}\,u^i_{\mathlarger\varepsilon}+C\|h_v^i\|_{p,\Omega},
\end{align} 
and similarly for the $\displaystyle{\inf_\Omega{u^i_{\mathlarger\varepsilon}}}.$ More more details, see Proposition 3.3 \cite{Caffar}. Using this along with the estimates \eqref{es1} and \eqref{es2}, we have the following: 
\begin{align}\label{eq 3.10}
	\|u^i_{\mathlarger\varepsilon}\|_{W^{2,p}(\Omega)}\leq \widetilde{C},
\end{align}
where $\widetilde{C}$ is a positive constant independent of $i.$
\end{proof}

%
%
	
	\noindent{\textbf{\tcbox{Step 4: Establish the existence of solution to \eqref{eq propp}}}}
	It further gives that $\{u^i_\mathlarger\varepsilon\}$ is uniformly bounded in $W^{2,p}(\Omega).$ Now, by reflexivity of $W^{2,p}(\Omega),$ $u^i_\mathlarger\varepsilon$ converges weakly in $W^{2,p}(\Omega).$ Moreover, since $p>\frac{N}{2}.$ Using the similar arguments as above, we have the existence of a subsequence such that $u^i_{\mathlarger\varepsilon} \To u_{\mathlarger\varepsilon,v}$ in the Lipschitz norm.
	As a consequence of Theorem \ref{Caffar}, $u_{\mathlarger\varepsilon,v}$ is a $W^{2,p}$-viscosity solution to (\ref{eq propp}).
	Moreover, $u_{\mathlarger\varepsilon,v}$ satisfies the following estimate:
	\begin{align}
		\|u_{\mathlarger\varepsilon,v}\|_{W^{2,p}(\Omega)}\leq C\bigg(\displaystyle{\max_{\partial\Omega}}~|g|+\|g\|_{W^{2,p}(\Omega)}+\|f(|u\geq u(x)|)\|_{p,\Omega}\bigg).
	\end{align}
	
	\noindent \textbf{\tcbox{Step 5: Establish the existence of solution to \eqref{eq v}}}  
	Further, using Theorem 2.1 \cite{Caffar1} (an application of Schaefer fixed point theorem), we have the existence of a $W^{2,p}$-viscosity solution to \eqref{eq v} for each $0<\mathlarger{\mathlarger\varepsilon}<1,$ say $u_{\mathlarger\varepsilon}.$ Moreover, $u_{\mathlarger\varepsilon}$ satisfies the following estimate:
	\begin{align}\label{eq u_e}
		\|u_{\mathlarger\varepsilon}\|_{W^{2,p}(\Omega)}\leq C\bigg(\displaystyle{\max_{\partial\Omega}}~|g|+\|g\|_{W^{2,p}(\Omega)}+\|f(|u\geq u(x)|)\|_{p,\Omega}\bigg).
	\end{align}

	\noindent{\textbf{\tcbox{Step 6: Establish the existence of solution to \eqref{eq 3.4} on $\mathlarger{\mathlarger\varepsilon}\To 0$}}}
	Since $u_{\mathlarger\varepsilon}$ is uniformly bounded in $W^{2,p}(\Omega)$ so we have that along some subsequence, $u_{\mathlarger\varepsilon}$ converges weakly in $W^{2,p}(\Omega).$ Moreover, by the Rellich-Kondrasov theorem, along some subsequence $u_{\mathlarger\varepsilon}\To u$ in $C(\overline{\Omega})$ to a Lipschitz function $u.$ We further claim that $u$ is an $L^p$-viscosity solution to (\ref{eq 3.4}). We use the idea of \cite{Brand}. We just check the supersolution part. Further, one can check for the subsolution part using the same arguments. Let, if possible, assume that $u$ is not an $L^p$-viscosity supersolution to (\ref{eq 3.4}). Then by definition, there exists a point $x_0\in\Omega$ and a function $\phi\in W^{2,p}(\Omega)$ with $D\phi\neq 0$ such that $u-\phi$ has local minimum at $x_0$ and 
	\begin{align}\label{eq 3.012}
		|D\phi|^\gamma \mathcal{M}_{\lambda,\Lambda}^+(D^2\phi)-f(|u\geq u(x_0)|)\geq \alpha \text{ a.e. in some ball }B(x_0,r),
	\end{align}
	for some constant $\alpha>0.$ In other words, $u-\phi$ restricted to $\overline{B(x_0,r)}$ has a global strict minima at $x_0.$ Next, using the above information, we get a contradiction by constructing a function $\phi_\mathlarger\varepsilon=\phi-\psi_\mathlarger\varepsilon$ corresponding to $u_\mathlarger\varepsilon$ such that
	\begin{align}\label{eq 3.018}
		|D\phi_\mathlarger\varepsilon|^\gamma \mathcal{M}_{\lambda,\Lambda}^+(D^2\phi_\mathlarger\varepsilon)+\mathlarger{\mathlarger\varepsilon}\Delta\phi_\mathlarger\varepsilon-f(|u\geq u(x_0)|)\geq \alpha \text{ a.e. in }B(x_0,r)
	\end{align}
	for small enough $\mathlarger{\mathlarger\varepsilon}>0$ and \begin{align}\label{eq 3.19}
		\phi_\mathlarger\varepsilon\To\phi \text{ uniformly}.
	\end{align}

	Now, since $u_\mathlarger\varepsilon$ is an $L^p$-viscosity solution to (\ref{eq v}) so (\ref{eq 3.018}) implies that $u_\mathlarger\varepsilon-\phi_\mathlarger\varepsilon$ can not attain minimum in the ball $B(x_0,r).$ However, since $u_\mathlarger\varepsilon-\phi_\mathlarger\varepsilon$ is continuous and $B(x_0,r)$ is compact. Therefore, $u_\mathlarger\varepsilon-\phi_\mathlarger\varepsilon$ attains minimum in $\overline{B(x_0,r)}.$ Let it be $x_\mathlarger\varepsilon.$ It gives that $x_\mathlarger\varepsilon\To x_0$ along some subsequence. It further implies that $x_\mathlarger\varepsilon \in B(x_0,r)$ for small enough $\mathlarger\varepsilon,$ which is a contradiction. Thus such a function $\phi$ constructed in \eqref{eq 3.012} does not exist, which proves our claim that $u$ is an $L^p$-supersolution to \eqref{eq 3.4}. Similarly, one can check the subsolution part.

	Next, we show that $u$ is the limit function of the sequence of functions $u_\mathlarger\varepsilon$ as $\mathlarger{\mathlarger\varepsilon}\To 0.$
	Let if possible, $\mathlarger{\mathlarger\varepsilon}_i$ and $\widetilde{\mathlarger{\mathlarger\varepsilon}}_i$ be two sequences approaching $0$ with $u$ and $\widetilde{u}$ being the corresponding limit functions to the sequences, respectively. Up to subsequences, we may assume that
	\begin{align}
		\cdots\leq \widetilde{\mathlarger{\mathlarger\varepsilon}}_{i+1}\leq \mathlarger{\mathlarger\varepsilon}_i\leq \widetilde{\mathlarger{\mathlarger\varepsilon}}_i\leq \mathlarger{\mathlarger\varepsilon}_{i-1}\leq \cdots.
	\end{align}  
	Our aim is to show that $w=u_{\mathlarger{\mathlarger\varepsilon}_i}-u_{{\widetilde{\mathlarger{\mathlarger\varepsilon}}}_{i+1}}\leq 0.$ If we show that $|Dw|^\gamma F(D^2w)+\mathlarger{\mathlarger\varepsilon}\Delta w\geq 0$ in $\Omega$ (in $C$-viscosity sense), we are done.  
	As by comparison principle, we would immediately get $w\leq 0.$ Therefore, $u_{\mathlarger{\mathlarger\varepsilon}_i}\leq u_{{\widetilde{\mathlarger{\mathlarger\varepsilon}}}_{i+1}}.$ 
	
	Thus, in order to show that $w=u_{\mathlarger{\mathlarger\varepsilon}_i}-u_{{\widetilde{\mathlarger{\mathlarger\varepsilon}}}_{i+1}}\leq 0,$ we only need to show that  $|Dw|^\gamma F(D^2w)+\mathlarger{\mathlarger\varepsilon}_i\Delta w\geq 0.$ As shown above, it immediately gives $w\leq 0.$ Let us assume the contrary, i.e., there exists some point $x_0\in\Omega$ such that for some $\varphi\in C^2(\Omega),$ $w-\varphi$ attains local maxima at $x_0,$ i.e., there exists a ball $B(x_0,r)$ such that
	\begin{align}
		|D\varphi|^\gamma \mathcal{M}_{\lambda,\Lambda}^+(D^2\varphi)+\mathlarger{\mathlarger\varepsilon}_i\Delta \varphi\leq -\alpha \text{ in }B(x_0,r),
	\end{align}
	for some $\alpha>0$ and $w-\varphi=\big(u_{\mathlarger{\mathlarger\varepsilon}_i}-u_{{\widetilde{\mathlarger{\mathlarger\varepsilon}}}_{i+1}}\big)-\varphi=u_{\mathlarger{\mathlarger\varepsilon}_i}-\big(u_{{\widetilde{\mathlarger{\mathlarger\varepsilon}}}_{i+1}}+\varphi\big)$ has a global strict maximum at $x_0$ in $B(x_0,r).$
	Now, consider the function $\Psi=\varphi+u_{\widetilde{{\mathlarger{\mathlarger\varepsilon}}}_{i+1}}.$ Clearly, $\Psi\in W^{2,p}(\Omega)$ and touches $u_{{\mathlarger{\mathlarger\varepsilon}}_i}$ from above at $x_0.$ Also, consider a test function, $\Phi$ for $u_{\widetilde{\mathlarger{\mathlarger\varepsilon}}_{i+1}}$ touching from below with $|D\Phi(x_0)|$ sufficiently larger than $|D\varphi(x_0)|.$ We have
	\begin{align}\label{ch}
		\\&|D\Psi(x_0)|^\gamma \mathcal{M}_{\lambda,\Lambda}^+\big(D^2\Psi(x_0)\big)+\mathlarger{\mathlarger\varepsilon}_i\Delta\Psi(x_0)-f\big(|u\geq u(x_0)|\big)+\alpha\\
		&\leq|D\Psi(x_0)|^\gamma \big(\mathcal{M}_{\lambda,\Lambda}^+\big(D^2\varphi(x_0)\big)+\mathcal{M}_{\lambda,\Lambda}^+\big(D^2\Phi(x_0)\big)\big)+\mathlarger{\mathlarger\varepsilon}_i\Delta\varphi(x_0)+\mathlarger{\mathlarger\varepsilon}_i\Delta \Phi(x_0)-f\big(|u\geq u(x_0)|\big)+\alpha\\
		&=|D\Psi(x_0)|^\gamma \mathcal{M}_{\lambda,\Lambda}^+\big(D^2\varphi(x_0)\big)+\mathlarger{\mathlarger\varepsilon}_i\Delta\varphi(x_0)+|D\Psi(x_0)|^\gamma \mathcal{M}_{\lambda,\Lambda}^+\big(D^2\Phi(x_0)\big)+\mathlarger{\mathlarger\varepsilon}_i\Delta \Phi(x_0)-f\big(|u\geq u(x_0)|\big)+\alpha\\
		&=|D\varphi(x_0)+D\Phi(x_0)|^\gamma \mathcal{M}_{\lambda,\Lambda}^+\big(D^2\varphi(x_0)\big)+\mathlarger{\mathlarger\varepsilon}_i\Delta\varphi(x_0)+\alpha\\
		&\qquad+|D\varphi(x_0)+D\Phi(x_0)|^\gamma \mathcal{M}_{\lambda,\Lambda}^+(D^2\Phi(x_0))+\mathlarger{\mathlarger\varepsilon}_i\Delta \Phi(x_0)-f\big(|u\geq u(x_0)|\big)\\
		&\leq \frac{|D\varphi(x_0)+D\Phi(x_0)|^\gamma}{|D\varphi(x_0)|^\gamma}{\big(-\alpha-\mathlarger{\mathlarger\varepsilon}_i\Delta\varphi(x_0)\big)}+\mathlarger{\mathlarger\varepsilon}_i\Delta\varphi(x_0)+\alpha\\
		&\qquad+\frac{|D\varphi(x_0)+D\Phi(x_0)|^\gamma}{|D\Phi(x_0)|^\gamma}\big(f\big(|u\geq u(x_0)|\big)-\mathlarger{\mathlarger\varepsilon}_i\Delta \Phi(x_0)\big)+\mathlarger{\mathlarger\varepsilon}_i\Delta \Phi(x_0)-f\big(|u\geq u(x_0)|\big)\\
		&\leq \frac{|D\varphi(x_0)+D\Phi(x_0)|^\gamma}{|D\varphi(x_0)|^\gamma}{\big(-\alpha-\mathlarger{\mathlarger\varepsilon}_i\Delta\varphi(x_0)\big)}+\mathlarger{\mathlarger\varepsilon}_i\Delta\varphi(x_0)+\alpha\\
		&\qquad+2^{\gamma-1}\frac{|D\varphi(x_0)|^\gamma+|D\Phi(x_0)|^\gamma}{|D\Phi(x_0)|^\gamma}\big(f\big(|u\geq u(x_0)|\big)-\mathlarger{\mathlarger\varepsilon}_i\Delta \Phi(x_0)\big)+\mathlarger{\mathlarger\varepsilon}_i\Delta \Phi(x_0)-f\big(|u\geq u(x_0)|\big)\\
		&\leq {\big(-\alpha-\mathlarger{\mathlarger\varepsilon}_i\Delta\varphi(x_0)\big)}\left(\frac{|D\varphi(x_0)+D\Phi(x_0)|^\gamma}{|D\varphi(x_0)|^\gamma}-1\right)\\
		&\qquad+\left(f\big(|u\geq u(x_0)|\big)-\mathlarger{\mathlarger\varepsilon}_i\Delta\Phi(x_0)\right)\left(2^{\gamma-1}\frac{\big(|D\varphi(x_0)|^\gamma+|D\Phi(x_0)|^\gamma\big)}{|D\Phi(x_0)|^\gamma}-1\right),
	\end{align}
for all large enough $i\in\mathbb{N}.$ Note that in the second last step we used the fact that for any positive real numbers $a,b$ and $r\geq 1,$ we have 
\begin{align}
|a+b|^r\leq 2^{r-1}\big(|a|^r+|b|^r\big).
\end{align} 
Further, by the choice of test function $\Phi$ made before \eqref{ch}, we have 
	\begin{align}
		F(D^2\Psi)+\mathlarger{\mathlarger\varepsilon}_i\Delta\Psi-f\big(|u\geq u(x_0)|\big)\leq -\alpha<0,		
	\end{align}
 which contradicts the fact that $u_{{{\mathlarger{\mathlarger\varepsilon}}}_{i}}$ is an $L^p$-viscosity solution to \eqref{eq v}. Thus we have that $u_{\mathlarger{\mathlarger\varepsilon}_i}\leq u_{{\widetilde{\mathlarger{\mathlarger\varepsilon}}}_{i+1}}.$ Letting $i\To\infty,$ we get $u\leq \widetilde{u}.$ Also, following the similar arguments, one can show that $u_{{\widetilde{\mathlarger{\mathlarger\varepsilon}}}_{i+1}}\leq u_{\mathlarger{\mathlarger\varepsilon}_i}.$ Thus, we have $\widetilde{u}\leq u$ and hence $u=\widetilde{u}.$\\ Therefore, we have the existence of a $W^{2,p}$-viscosity solution, $u$ to \eqref{eq 3.4}. Moreover, by \eqref{eq u_e}, we have the following estimate:
	\begin{align}
		\|u\|_{W^{2,p}(\Omega)}\leq C\bigg(\|u\|_{\infty,\Omega}+\|g\|_{W^{2,p}(\Omega)}+\|f(|u\geq u(x)|)\|_{p,\Omega}\bigg). 
	\end{align}\qed
	
	\section{\textbf{Funding and/or Conflicts of interests/Competing interests}}
My sincere thanks go out to Prof. Jagmohan Tyagi for his constructive remarks and suggestions. The research was financially supported by Council of Scientific $\&$ Industrial Research (CSIR) under the grant no. 09/1031(0005)/2019--EMR--I.\\
	
There are no conflict of interests of any type. This manuscript does not use any kind of data.	
	

\begin{thebibliography}{90}
	    \bibitem{Brand} Amadori AL, Brandolini B, Trombetti C. Viscosity solutions of the Monge-Amp\'ere equation with the right hand side in $L^p.$ Rend. Lincei Mat. Appl. 2007; 18: 221--233.
		\bibitem{C1beta} Birindelli I, Demengel F. $C^{1,\beta}$ regularity for Dirichlet problems associated to fully nonlinear degenerate elliptic equations. ESAIM Control Optim. Calc. Var. 2014; 20: 1009--1024.
		\bibitem{Birin 1} Birindelli I, Demengel F. Comparison principle and Liouville type results for singular fully nonlinear operators. Ann. Fac. Sci. Toulouse Math. 6. 2004; 13(2): 261--287.
		\bibitem{Birin 3} Birindelli I, Demengel F. Eigenvalue, maximum principle and regularity for fully nonlinear homogeneous operators. Comm. Pure Appl. Analysis. 2007; 6(2): 335--366.
		\bibitem{Birin 2} Birindelli I, Demengel F. Regularity and uniqueness of the first eigenfunction for singular, fully nonlinear elliptic operators. J. Differential Equations. 2010; 249(5): 1089--1110.
		\bibitem{Birin 4} Birindelli I, Demengel F. Regularity for radial solutions of degenerate fully nonlinear equations. Nonlinear Anal. 2012; 75(17): 6237--6249. 
		\bibitem{Caffar} Caffarelli L, Crandall MG, Kocan M, Swiech A. On viscosity solutions of fully nonlinear equations with measurable ingredients. Communications on Pure and Applied Mathematics. 1993; 49(4): 365--398.
		\bibitem{Caffar1} Caffarelli L, Tomasetti I. Fully nonlinear equations with application to Grad equations in plasma physics. Communications on Pure and Applied Mathematics. 2023; 76(3): 604--615.
		\bibitem{Crand} Crandall MG, Lions P.-L. Viscosity solutions of Hamilton-Jacobi equations. Trans. Amer. Math. Soc. 1983; 277: 1--42.
\bibitem{Felmer 18} Felmer P, Quaas A. On critical exponents for the Pucci’s extremal operators. Ann. Inst. H. Poincar\'e Anal. Nonlin\'eaire. 2003; 20(5): 843--865.
\bibitem{Felmer 19} Felmer P, Quaas A. Positive radial solutions to a \enquote{semilinear} equation involving the Pucci’s operator. J. Differential Equations. 2004; 199(2): 376--393.
		\bibitem{Grad} Grad H. Alternating dimension plasma transport in three dimensions, Technical report. New York Univ. NY (USA): Courant Inst. of Mathematical Sciences. 1979.
		\bibitem{Imbert} Imbert C, Silvestre L. $C^{1,\alpha}$ regularity of solutions of some degenerate, fully non-linear elliptic equations. Adv. Math. 2013; 233: 196--206.
		\bibitem{Ishii} Ishii H, Lions P.-L. Viscosity solutions of fully nonlinear second-order elliptic partial
		differential equations. J. Differential Equations. 1990; 83: 26--78.
		\bibitem{Laure1} Laurence P, Stredulinsky E. A new approach to queer differential equations. Communications on Pure and Applied Mathematics. 1985; 38(3): 333--355.
		\bibitem{Laure2} Laurence P, Stredulinsky E. A bootstrap argument for grad generalized differential equations. Indiana University Mathematics Journal. 1989; 38(2): 377--415.
		\bibitem{Temam2} Mossino J, Temam R. Directional derivative of the increasing rearrangement mapping and application to a queer differential equation in plasma physics. Duke Mathematical Journal. 1981; 48(3): 475--495.
		\bibitem{Oza Hopf} Oza P, Tyagi J. Probabilistic method to solution representation formula $\&$ Hopf lemma to Pucci's equation. 2022. pp. 13: Submitted for publication.
		\bibitem{Oza Gradient Degeneracy} Oza P, Tyagi J. Regularity of solutions to variable-exponent degenerate mixed fully nonlinear local and nonlocal equations. 2023. pp. 26: (\href{https://arxiv.org/abs/2302.06046}{arXiv link}) Submitted for publication.
		\bibitem{Quaas} Quaas A. Existence of a positive solution to a \enquote{semilinear} equation involving Pucci’s operator in a convex domain. Differential Integral Equations. 2004; 17(5--6): 481--494.
		\bibitem{Quaas 2} Quaas A, Sirakov B. Existence results for nonproper elliptic equations involving the Pucci operator. Comm. Partial Differential Equations. 2006; 31(7--9):  987--1003.
		\bibitem{Temam} Temam R. Monotone rearrangement of a function and the grad-mercier equation of plasma physics. In Proc. Int. Conf. Recent Methods in Nonlinear
		Analysis and Applications E. de Giogi and U. Mosco Eds. 1978.
		\bibitem{Tyagi 3} Tyagi J, Verma RB. Lyapunov-type inequality for extremal Pucci’s equations. J. Aust. Math. Soc. 2020; 109(3): 416--430.
		\bibitem{Tyagi 1} Tyagi J, Verma RB. Positive solution of extremal Pucci’s equations with singular and sublinear nonlinearity. Mediterr. J. Math. 2017; 14(148).
		\bibitem{Tyagi 2} Tyagi J, Verma RB. Positive solution to extremal Pucci’s equations with singular and gradient nonlinearity. Discrete Contin. Dyn. Syst. 2019; 39(5): 2637--2659.
		\bibitem{Winter} Winter N. $W^{2,p}$- and $W^{1,p}$-estimates at the boundary for solutions of fully nonlinear, uniformly elliptic equations. Zeitschrift f\"ur Analysis and ihre Anwendungen. 2009; 28(2): 129--164.  
	\end{thebibliography}
\end{document}